\documentclass[12pt]{amsart}

\def\timesover#1#2#3{\ \xymatrix@1@=0pt@M=0pt{ _{#1}&\times&_{#2} 
\\& ^{#3}&}\ }
\def\otimesover#1#2#3{\ \xymatrix@1@=0pt@M=0pt{ 
_{#1}&\otimes&_{#2} \\& ^{#3}&}\ }
\usepackage[all]{xy}

\newtheorem{thm}{Theorem}
\newtheorem{lem}[thm]{Lemma}

\newtheorem{prop}[thm]{Proposition}
\newtheorem{rmk}[thm]{Remark}
\numberwithin{equation}{section}

\newcommand{\Pic}{{\rm Pic}}

\newcommand{\G}{{\mathbb G}}
\newcommand{\Z}{{\mathbb Z}}
\renewcommand{\P}{{\mathbb P}}
\newcommand{\C}{{\mathbb C}}

\begin{document}

\title[Brauer group of the Moduli spaces of ${\rm PGL}_r({\mathbb 
C})$--bundles]{Unramified Brauer group of the moduli spaces of ${\rm PGL}_r({\mathbb C})$--bundles over 
curves}

\author[I. Biswas]{Indranil Biswas}

\address{School of Mathematics, Tata Institute of Fundamental
Research, Homi Bhabha Road, Bombay 400005, India}

\email{indranil@math.tifr.res.in}

\author[A. Hogadi]{Amit Hogadi}

\address{School of Mathematics, Tata Institute of Fundamental
Research, Homi Bhabha Road, Bombay 400005, India}

\email{amit@math.tifr.res.in}

\author[Y. I. Holla]{Yogish I. Holla}

\address{School of Mathematics, Tata Institute of Fundamental
Research, Homi Bhabha Road, Bombay 400005, India}

\email{yogi@math.tifr.res.in}



\date{}

\begin{abstract}
Let $X$ be an irreducible smooth complex projective curve of 
genus $g$, with $g\,\geq\, 2$. Let $N$ be a connected component of the moduli space of 
semistable principal ${\rm PGL}_r(\mathbb C)$--bundles over $X$; it is a normal
unirational complex projective variety.  We prove that the Brauer group of a desingularization of $N$ is 
trivial. 
\end{abstract}

\maketitle

\section{Introduction}

Let $X$ be an irreducible smooth complex projective curve, with
$\text{genus}(X)\, =\, g\, \geq\, 2$. For a fixed line bundle
$\mathcal L$ over $X$, let $M_X(r, {\mathcal L})$ 
be the coarse moduli space of semistable vector bundles over $X$ of rank 
$r$ and determinant $\mathcal L$. It is a normal 
unirational complex projective variety, and if $\text{degree}({\mathcal 
L})$ is coprime to $r$, then $M_X(r, {\mathcal L})$ is known to
be rational \cite{Ne}, \cite{KS}. Apart from these coprime case, and
the single case of $g\,=\, r\, =\,\text{degree}(\mathcal L)\,=\,2$
when $M_X(r, {\mathcal L})\,=\, \P^3_{\C}$,
the rationality of $M_X(r, {\mathcal L})$ 
is an open question in every other case. See \cite{Ho} for 
rationality of some other types of moduli spaces associated to $X$.

We consider the coarse moduli space $N_X(r,d)$ of  
semistable principal ${\rm PGL}_r(\mathbb C)$--bundles of topological type $d$ over $X$. Recall that a ${\rm PGL}_r(\mathbb C)$ bundle $P/X$ is said to be of topological type $d$ if the associated $\P^{r-1}$-bundle is isomorphic to ${\mathbb P}{\rm roj}({\mathbb E})$ for some rank $r$ vector bundle $E$ whose degree is congruent to $d$ modulo $r$.
This $N_X(r,d)$ is an irreducible normal unirational 
complex projective variety. This paper is a sequel to \cite{BHH}, where we investigate the Brauer group of
desginularization of moduli spaces attached to curves. This Brauer group is a birational invariant of the space and its vanishing is a necessary condition for the space involved to be rational. 

In this paper we prove that the Brauer group of a desingularization 
of $N_X(r,d)$ is zero (see Theorem \ref{theorem}). 

When $g\,=\, 2$, the moduli space $N_X(2,0)$ is a 
quotient of $\P^3_{\C}$ by a faithful action of the
abelian group $(\Z/2\Z)^4$. In this special case it follows the quotient is rational.

\section{Preliminaries}\label{pril}

We continue with the above set--up and notation.
Let $N_X(r,d)$ denote the coarse moduli space of S--equivalence
classes of all semistable principal 
$\text{PGL}_r(\mathbb C)$--bundles of topological type $d$ over $X$. For notational 
convenience, $N_X(r,d)$ will also be denoted by $N$.

Let ${M}_X(r,{\mathcal L}_X)$ be the coarse moduli space of 
S-equivalence classes of semistable vector bundles over $X$ of 
rank $r$ and determinant ${\mathcal L}_X$. 
Let $\Gamma$ be the group of all isomorphism classes
of algebraic line bundles $\tau$ over $X$ such that
$\tau^{\otimes r}\, =\, {\mathcal O}_X$. This group $\Gamma$ 
has the following natural action on ${M}_X(r,{\mathcal L}_X)$: 
the action of any $\tau\, \in\, \Gamma$ sends any $E\, \in\, 
{M}_X(r,{\mathcal L}_X)$ to $E\, \otimes \, \tau$. The moduli 
space $N$ is identified with the quotient variety
${M}_X(r,{\mathcal L}_X)/\Gamma$. Let
\begin{equation}\label{4f}
f\,:\,{M}_X(r,{\mathcal L}_X)\, \longrightarrow\, 
{M}_X(r,{\mathcal L}_X)/\Gamma\,=\, N
\end{equation}
be the quotient morphism.

For notational convenience, the moduli space $M_X(r,{\mathcal 
O}_X)$ will also be denoted by $M_{{\mathcal L}_X}$.

Let
\begin{equation}\label{ress}
M_{{\mathcal L}_X}^{\rm st}\,\subset\, M_{{\mathcal L}_X}
~\, ~ \text{ and }~\, ~ N^{\rm st}\,\subset\, N
\end{equation}
be the loci of stable bundles. The above action of
$\Gamma$ on $M_{{\mathcal O}_X}$ 
preserves $M_{{\mathcal O}_X}^{\rm st}$, and
$$
f(M_{{\mathcal O}_X}^{\rm st})\,=\, N^{\rm st}\, .
$$ 

\section{The action of $\Gamma$}

Consider the action of $\Gamma$ on $M_{{\mathcal O}_X}$ defined 
in Section \ref{pril}. 
For any primitive $\tau\, \in\, \Gamma$, i.e. an element of order $r$, let
\begin{equation}\label{is}
M^\tau_{{\mathcal O}_X}\, =\, \{E\, \in\, M_{{\mathcal O}_X}\, 
\mid\, E\otimes \tau\, =\, E\}\, \subset\, M_{{\mathcal O}_X}
\end{equation}
be the fixed point locus.

Take any nontrivial line bundle $\tau\, \in\, \Gamma$ of order $r$.
Let
\begin{equation}\label{res-phi}
\phi\, :\, Y \, \longrightarrow\, X
\end{equation}
be the \'etale cyclic covering of degree $r$ given by $\tau$. We 
recall the construction of $Y$ as the spectral cover associated to the equation $\tau ^r \cong {\mathcal O}_X$.

Let
$$
\beta\, :\, Y\, \longrightarrow\, Y
$$
be a nontrivial generator  of the Galois group 
$\text{Gal}(\phi)\,=\, {\mathbb Z}/r {\mathbb Z}$ of the 
covering $\phi$. The homomorphism
$\xi\, \longmapsto\, \beta^*\xi$ defines an action of 
$\text{Gal}(\phi)$ on $\text{Pic}^d(Y)$ for any $d$.

Let 
\begin{equation}\label{pull-back}
\phi^*:\, \Pic^0(X) \, \longrightarrow \, \Pic^0(Y)
\end{equation}
be the pullback homomorphism $L\, \longmapsto\, \phi^*L$. Let 
$K$ denote the kernel of $\phi^*$; it
is a group of order $r$ generated by $\tau$. Let 
\begin{equation}\label{norm0}
{\rm Nm}: \, \Pic^d (Y) \, \longrightarrow\, \Pic^d(X) 
\end{equation}
 and 
 \begin{equation}\label{norm}
{\rm N}: \, \Pic^d (Y) \, \longrightarrow\, \Pic^d(X) 
\end{equation} 
 be the norm homomorphism and the twisted norm morphism. We recall that ${\rm Nm}$ takes a line bundle $\xi$ to the descent of $\otimes (\beta^{*i}{\xi})$ and 
  ${\rm N}$ sends a line bundle $\xi$ to ${\rm Nm}(\xi)\otimes \tau ^{(r(r-1)/2}$. 

The group $\Gamma$ has a natural action on $\Pic^d (Y)$; any 
$\sigma\, \in\, \Gamma$ acts as the automorphism $\xi\, 
\longmapsto\, \xi\otimes \phi^*\sigma$. Therefore, $\phi^*$
in \eqref{pull-back} is $\Gamma$--equivariant, and the
kernel $K$ acts trivially on ${\rm Pic}^d(Y)$. The morphism
${\rm N}$ in \eqref{norm} factors through the quotient morphism
${\rm Pic}^d(Y)\,\longrightarrow\,\Pic^d(Y)/\Gamma$. The action of
$\Gamma$ on $\Pic^d (Y)$ clearly commutes with the action of 
$\text{Gal}(\phi)$ defined earlier. 

Let
\begin{equation}\label{recu}
{\mathcal U}_{{\mathcal L}_X}\, :=\,{\rm N}^{-1}({\mathcal L}_X)\setminus
({\rm N}^{-1}({\mathcal L}_X))^{{\rm Gal}(\phi)}
\, \subset \, {\rm Nm}^{-1}({\mathcal L}_X)
\end{equation}
be the complement of the fixed point locus for the
action of $\text{Gal}(\phi)$. It is a
$\Gamma$--invariant open subscheme. 

Now we state a well-known result (cf. \cite[Lemma 3.4]{NR}).

\begin{lem}\label{BNR}
Take any primitive line bundle $\tau\, \in\, \Gamma$.
The reduced closed subscheme $$(M^{\rm st}_{{\mathcal 
L}_X})^\tau\,:=\,M^{\rm
st}_{{\mathcal L}_X}\cap M^{\tau}_{{\mathcal L}_X}\, 
\subset\, M^{\rm st}_{{\mathcal L}_X}$$
(see \eqref{is} and \eqref{ress}) is 
$\Gamma$--equivariantly isomorphic to the quotient scheme
$$
{\mathcal U}_{{\mathcal L}_X}/{\rm Gal}(\phi)\, .
$$
\end{lem}

\begin{lem}\label{HP}
The norm map as defined in \eqref{norm} is surjective, and
there is a bijection of the set of connected components
$\pi_0({\rm N}^{-1}({\mathcal L}_X))$ with the Cartier dual 
$K^{\vee}\, :=\, {\rm Hom}(K,\, {\mathbb C}^*)\,=\,
{\mathbb Z}/r\mathbb Z$, where $K\,:=\, {\rm kernel}(\phi^*)\,
=\, {\mathbb Z}/r{\mathbb Z}$.
\end{lem}

Lemma \ref{HP} is proved in \cite{NR} (see \cite[Proposition 
3.5]{NR}). 

Let $V_0$ be the connected components of ${\rm N}^{-1}(
{\mathcal O}_X)$, with ${\mathcal O}_Y\,\in\, V_0$. Since
${\rm Nm}^{-1}({\mathcal O}_X)$ is smooth, both $V_0$.
is irreducible.

\begin{lem}\label{gal}
The action of ${\rm Gal}(\phi)$ on ${\rm 
N}^{-1}({\mathcal O}_X)$ preserves the connected component $V_0$. 
For the action of ${\rm Gal}(\phi)$ on ${\rm 
N}^{-1}({\mathcal L}_X)$ the quotient ${\rm 
N}^{-1}({\mathcal L}_X)/{\rm Gal}(\phi)$ has exactly $(r,d)$ components which are smooth, here $(r,d)$ is the greatest common divisor of $r$ and $d$.
\end{lem} 

\begin{proof}
The point ${\mathcal O}_Y\, \in\, \text{Pic}^0(Y)$ is 
fixed by ${\rm Gal}(\phi)$; hence $V_0$ is fixed by
${\rm Gal}(\phi)$. Therefore, the other component, namely $V_1$, is
also fixed by ${\rm Gal}(\phi)$. See \cite{NR}, Proposition 3.5 for the proof of the second statement.
\end{proof}

\begin{lem}\label{nss} Let $r=2$. 
The set of all points in the complement $M_{{\mathcal O}_X}\setminus 
M^{\rm st}_{{\mathcal O}_X}$ (see \eqref{ress}) fixed by $\tau$ is finite. 
\end{lem}

\begin{proof}
Take any point $x\,\in\, M_{{\mathcal O}_X}\setminus M^{\rm 
st}_{{\mathcal O}_X}$. Let $E\,= \, L\oplus L^*$, with $L\, 
\in\, \text{Pic}^0(X)$, be the unique 
polystable vector bundle representing 
the point $x\,\in\, M_{{\mathcal O}_X}$. The action of $\tau$
takes the point $x$ to the point represented by the polystable 
vector bundle $(L\otimes 
\tau) \oplus (L^*\otimes\tau)$.

Assume that $\tau\cdot x\,=\,x$. Then 
the two vector bundles $L\oplus L^*$ and $(L\otimes \tau) \oplus 
(L^*\otimes\tau)$ are isomorphic. This implies that 
\begin{equation} \label{equal}
L\otimes\tau \, \cong \, L^*
\end{equation}
(recall that $\tau$ is nontrivial; so $L\,\not=\, L\otimes\tau$). From
\eqref{equal} it follows that $L^{\otimes 2}\,=\, (L^{\otimes 2})^*$.

Consequently, isomorphism classes of all line bundles $L\, \in\,
\text{Pic}^0(X)$ satisfying \eqref{equal}, for a given
$\tau$, is a finite subset. Therefore, there
are only finitely many points of $M_{{\mathcal O}_X}\setminus M^{\rm
st}_{{\mathcal O}_X}$ that are fixed by $\tau$.
\end{proof}

\begin{rmk}\label{g2}
{\rm When genus of $X$ equals $2$, ${\rm dim}({\rm Pic}^0(Y))=3$. It follows from Lemma \ref{BNR}, 
that $(M^{st}_{{\mathcal O}_X})^{\tau}$ is one dimensional and hence Lemma \ref{nss} implies that 
$M_{{\mathcal O}_X}^{\tau}$ is of codimension two in $M_{{\mathcal O}_X} \cong \P^3_{\C}$.
}
\end{rmk}

Let $\sigma\, \in \Gamma$ be another primitve  element such that $\sigma$ and $\tau$ are ${\mathbb Z}/r{\mathbb Z}$ linearly independent. The subgroup of $\Gamma$ generated $\sigma$ and $\tau$ will
be denoted by $A$. So $A$ is isomorphic to
$({\mathbb Z}/r{\mathbb Z})^{\oplus 2}$.

We note that $\Gamma\subset \Pic(X)$ is identified with
$H^1(X,\, {\mathbb Z}/r{\mathbb Z})\subset H^1(X,\G_m)$ under the 
natural inclusion. Let
\begin{equation}\label{rese}
e\, :\, \Gamma\otimes \Gamma\, \longrightarrow\, {\mathbb Z}/r\mathbb Z
\end{equation}
be the pairing given by the cup product
$$
H^1(X,\, {\mathbb Z}/r{\mathbb Z})\otimes H^1(X,\, {\mathbb Z}/r{\mathbb Z})
\, \stackrel{\cup}{\longrightarrow}\,
H^2(X,\, {\mathbb Z}/r{\mathbb Z})\,=\, {\mathbb Z}/
r\mathbb Z\, .
$$
It is known that this $e$ coincides with the Weil pairing
(see \cite[p. 183]{Mu}).

\begin{prop}\label{fixed-points}
Let $\sigma$ and $\tau$ be two primitive elements of $\Gamma$ such that they generate a subgroup $A=({\mathbb Z}/r{\mathbb Z})^{\oplus 2}$.
If the pairing $e(\sigma\,, \,\tau)\, =\, 0$ then there exists  a nonempty closed irreducible $A$--invariant subset of 
$M_{{\mathcal O}_X}^{\rm st}$ which is fixed pointwise by $\tau$.
\end{prop}

\begin{proof}
By \cite{BP2}, Proposition 4.5, it follows that under the condition $e(\sigma, \tau)=0$ there is a stable bundle $E$ of rank $r$ and determinant ${\mathcal O}_X$ such that 
$E\otimes \sigma=E\otimes \tau=E$. 

The condition $E\otimes \tau=E$ implies the existence of a line bundle $\xi \in {\rm N}^{-1}({\mathcal O}_X)$ such that $ \phi _*(\xi)=E$.
The condition $E\otimes \sigma =E$ implies that there is a $\beta \in \text{Gal}(\phi)$ such that $\xi \otimes \phi ^*\sigma= \beta ^*\xi$.
Hence $\phi^*\sigma = (\beta^*\xi) \otimes \xi^{-1}$ lies in $V_0$, because $\beta^*\xi$ and $\xi^{-1}$ lie in the
same component (see both parts of the  Lemma \ref{gal}).

One observes that ${\rm N}^{-1}({\mathcal L}_X)$ is $\Gamma$ equivariantly isomorphic to the translate $L \cdot {\rm Nm}^{-1}({\mathcal O}_X)$ by any line bundle $L$ such that $N(L)={\mathcal L}_X$. This along with the above fact that  $\phi^*\sigma \in V_0$ implies that the translation by  $\phi^*\sigma$ preserves the connected 
components of ${\rm N}^{-1}({\mathcal L}_X)$.

Hence we conclude that  any connected component of the
quotient ${\mathcal U}_{{\mathcal L}_X}/{\rm Gal}(\phi)\, \subset\, (M^{\rm st}_{{\mathcal O}_X})^\tau$ is a
closed irreducible $A$-invariant subscheme of $M_{{\mathcal O}_X}^{\rm st}$
which is fixed pointwise by $\tau$.
This completes the proof of the proposition.
\end{proof}

\section{Brauer group of a desingularization of $N$}

In this section we identify the second cohomology $H^2(\Gamma,\,{\mathbb C}^*)$ with the space of alternating 
bi-multiplicative maps from $\Gamma$ to ${\mathbb C}^*$ (see
\cite[p. 215, Proposition 4.3]{Ra}); the group $H^2(\Gamma,\, 
{\mathbb C}^*)$ is isomorphic to the dual of the second exterior 
power of $(\Z/2\Z)^{2g}$. 

Recall that under the identification of $\Gamma$ with $H^1(X, \Z/r\Z)$, the Weil pairing
coincides with the intersection pairing.  Let $\{a_1,b_1,\cdots a_g,b_g\}$ be a symplectic basis for $H^1(X, \Z/r\Z)$.
in other words we have $e(a_i,a_j)=0=e(b_i,b_j)$ for all $i$ and $j$,  and $e(a_i,b_j)=\delta _{i,j}$.

Let $G\, \subset\, H^2(\Gamma,\,{\mathbb C}^*)$ be defined by 
$$
G\,:= \, \{ b\, \in\, H^2(\Gamma,\,{\mathbb C}^*)\, \mid \,
e(\sigma_1\, ,\,\sigma_2)\, =\, 0\, \Rightarrow\,
b(\sigma_1\, ,\,\sigma_2)\,= \,0\}\, .
$$
Let $H$ be the subgroup of $G$ of order two generated by the Weil pairing $e$.

\begin{lem}\label{mistake}
The group $G$ coincides with the subgroup $H$.
\end{lem}
\begin{proof}
Fix an $i$ and $j$ such that $i\neq j$ then one checks that $e(a_i+b_j,a_j-b_i)=e(a_j,b_j)-e(a_i,b_i)=0$ hence if $f \in G$  
then $0=f(a_i+b_j,a_j-b_i)=f(a_j,b_j)-f(a_i,b_i)$. This implies that $f$ is a multiple of $e$. 
\end{proof}

Our main theorem is the following.
\begin{thm}\label{theorem}
Let ${\widehat N}$ be a desingularization of the moduli space $N$. Then the Brauer group
${\rm Br}({\widehat N})=0$ 
\end{thm}

\begin{proof} We first assume that either $g\,\geq \,3$ or when $g=2$ rank $r>2$. The case of $g\,=\,2$ and rank $r=2$ 
will be treated separately.

It is enough to prove the theorem for some desingularization $\widehat{N}$ of 
$N$ because the Brauer group is a birational invariant for the 
smooth projective varieties. We choose a $\Gamma$--equivariant desingularization
\begin{equation}\label{p}
p\,:\, {\widetilde M}_{{\mathcal O}_X}\,\longrightarrow 
\,M_{{\mathcal O}_X}
\end{equation}
which is an isomorphism over $M^{\rm st}$; so ${\widetilde 
M}_{{\mathcal L}_X}$ is equipped with an action of $\Gamma$ 
given by the action of $\Gamma$ on $M_{{\mathcal L}_X}$. Define
$$
{\widetilde N}\, :=\, {\widetilde M}_{{\mathcal L}_X}/\Gamma\, .
$$

Let
$$
{\widehat N}\, \longrightarrow\, {\widetilde N}
$$
be a desingularization of ${\widetilde N}$ which is an isomorphism over
the smooth locus. So ${\widehat 
N}$ is also a desingularization of $N$. 

A stable principal $\text{PGL}_r({\mathbb C})$--bundle $E$ on $X$ 
is called \textit{regularly stable} if
$$
\text{Aut}(E)\,=\, e
$$
(by $\text{Aut}(E)$ we denote the automorphisms of the 
principal bundle $E$ over the identity map of $X$).
It is known that the locus of regularly stable bundles in $N$, 
which we will denote by $N^{\rm rst}$, coincides with the smooth 
locus of $N$ \cite[Corollary 3.4]{BHof}. Define
$$
M^{\rm rst}\, :=\, f^{-1}(N^{\rm rst})\, ,
$$
where $f$ is the morphism in \eqref{4f}. We note that the action of
$\Gamma$ on $M_{{\mathcal L}_X}$ preserves $M^{\rm rst}$, because 
$f$ is an invariant for the action of $\Gamma$. The 
action of $\Gamma$ on $M^{\rm rst}$ can be shown to be free. 
Indeed, if $E\, =\, E\otimes\tau$, where $\tau$ is nontrivial,
any isomorphism of $E$ with $E\otimes\tau$ produces a nontrivial
automorphism of ${\mathbb P}(E)$, because ${\mathbb 
P}(E\otimes\tau)\,=\, {\mathbb P}(E)$. Hence such a vector bundle $E$ 
cannot lie in $M^{\rm rst}$.

Consequently, the projection $f$ in \eqref{4f} defines a 
principal $\Gamma$--bundle
\begin{equation}\label{gfb}
M^{\rm rst}\, \stackrel{f}{\longrightarrow}\, N^{\rm rst}\, .
\end{equation}
Since $N$ is normal, and $N^{\rm rst}$ is its smooth locus, it 
follows that the codimension of the complement $N\setminus
N^{\rm rst}$ is at least two. Therefore, the codimension of the 
complement of $M^{\rm rst} \,\subset \, M_{{\mathcal O}_X}$ is at 
least two. Hence 
$$ H^0(M^{\rm rst},{\mathbb G}_m)={\mathbb C}^* $$
The Serre spectral sequence for the above principal 
$\Gamma$--bundle gives an exact sequence
$$
\text{Pic}(N^{\rm rst}) \,\stackrel{\delta}{\longrightarrow}\, 
\text{Pic}(M^{\rm 
rst})^\Gamma\, \longrightarrow\, H^2(\Gamma,\, {\mathbb C}^*)\, .
$$
We have $\text{Pic}(M^{\rm rst})^\Gamma/{\rm image}(\delta)
\,=\, \Z/l\Z$ \cite{BH} (see (3.5) in \cite{BH} and 
lines following it) where $l=(r,d)$. Hence we get an inclusion
\begin{equation}\label{res-go2}
\Z/l\Z\, \hookrightarrow\, H^2(\Gamma,\, {\mathbb C}^*)\, ,
\end{equation}
where the generator of $\Z/l\Z$ maps to the Weil pairing $e$
(see the proof of Proposition 9.1 in \cite[p. 203]{BLS}).

For the chosen desingularization $\widehat{N} \to N$, we have
\begin{equation}\label{resem}
{\rm Br}({\widehat N}) \,\subset\, {\rm Br}(N^{\rm rst})\,=\,
{\rm Br}(M^{\rm rst}/\Gamma)
\end{equation}
using the inclusion of $N^{\rm rst}$ in ${\widehat N}$. The 
Brauer group ${\rm Br}(N^{\rm rst})$ is computed in \cite{BH}.

The Serre spectral sequence for the principal 
$\Gamma$--bundle in \eqref{gfb} gives the following exact 
sequence:
\begin{equation}\label{rho}
H^2(\Gamma,\, {\mathbb C}^*)\,\stackrel{\rho}{\longrightarrow} 
\,H^2(M^{\rm 
rst}/\Gamma,\,\G_m)\,\longrightarrow\,H^2(M^{\rm rst},\,\G_m)\, .
\end{equation}
Let $\mathbb S$ be the set of all bicyclic subgroups $A\,
\subset\, \Gamma$  of the form $(\Z/r\Z)^{\oplus 2}$ satisfying the condition that there is some closed
irreducible subvariety $\mathcal Z$ of ${\widetilde M}_{{\mathcal O}_X}$ preserved
be the action of $A$ such that a primitive element of $A$ fixes $\mathcal Z$. 

Define the subgroup
$$G'\, :=\, \bigcap_{A\in{\mathbb S}}{\rm kernel}(H^2(\Gamma,\,
{\mathbb C}^*) \,\rightarrow\,H^2(A,\,{\mathbb C}^*))\, \subset
\, H^2(\Gamma,\,{\mathbb C}^*)\, .
$$ 
Using a theorem of Bogomolov, \cite[p. 288, Theorem 1.3]{Bo}, we have
\begin{equation}\label{i}
\rho^{-1}(H^2({\widehat N},\,\G_m))\,\subset\,G'
\end{equation}
(see \eqref{resem}).

We will show that $G'$ is a subgroup of  $G$ in $H^2(\Gamma,\,
{\mathbb C}^*)$. This will prove that the image $\rho( \rho^{-1}(H^2({\widehat N},\,\G_m)))=0$

$b\, \in\, G'$. We need to check that
$b(\sigma\, ,\,\tau)\, =\, 0$ whenever $e(\sigma\, ,\,\tau)\, =\, 0$ for a pair of primitive elements generating the subgroup $A=\Z/r\Z^{\oplus 2}$.
Since $e(\sigma\, ,\,\tau)\, =\, 0$, by Proposition \ref{fixed-points},
there is an irreducible closed subscheme $Z\, \subset\, M_{{\mathcal O}_X}^{\rm st}$
which is $A$--invariant and fixed pointwise by $\tau$. 

Since the $\Gamma$--equivariant desingularization $p$ in \eqref{p} is an isomorphism
over $M_{{\mathcal O}_X}^{\rm st}$, we conclude that the closure of $Z$ in
${\widetilde M}_{{\mathcal O}_X}$ is an $A$-invariant closed irreducible subscheme which is
fixed pointwise by $\tau$. Hence the action of $A$ on this closure is cyclic.
This implies that $A\, \in\, {\mathbb S}$, and hence $b(\sigma\, ,\,\tau)\, =\, 0$.
Therefore $G'\, \subset\, G$. 
 
Consequently, we have shown that $ H^2({\widehat N},\,\G_m) \cap {\rm image}(\rho)=0$.

This proves that the composition
\begin{equation}\label{c1}
H^2({\widehat N},\, \G_m) \,\longrightarrow\, H^2(M^{\rm 
rst}/\Gamma,
\, \G_m)\,\longrightarrow \,H^2(M^{\rm rst},\, \G_m)
\end{equation}
is injective. We will prove that this composition 
is zero (these homomorphisms are induced by the inclusion
$M^{\rm rst}/\Gamma\, \hookrightarrow\, \widehat N$ and the
quotient map to $M^{\rm rst}/\Gamma$).

Consider the diagram 
$$ \xymatrix{ 
{\widehat M} \ar[r]\ar[d] & 
{\widetilde M}_{{\mathcal L}_X} \ar[d] \\
{\widehat N}\ar[r] & {\widetilde N}}
$$
where ${\widehat M}$ is a $\Gamma$--equivariant desingularization 
of the closure of
$$
M^{\rm rst}\,=\, {\widehat N}\times _{N^{\rm rst}} M^{\rm rst}$$ 
in the fiber product ${\widehat N}\times_{{\widetilde 
N}}{\widetilde M}_{{\mathcal L}_X}$. This gives an action of 
$\Gamma$ on the smooth projective variety 
${\widehat M}$ which has a $\Gamma$--invariant open
subscheme $M^{\rm rst}$
with the quotient $M^{\rm rst}/\Gamma$ 
being the Zariski open subset $N^{\rm rst}$ of ${\widehat N}$.
Using the commutativity of $\Gamma$--actions we obtain a commutative 
diagram of homomorphisms
\begin{equation}\label{dg}
\xymatrix{ 
H^2({\widehat N}, \G_m) \ar[r]\ar[d] & 
H^2(N^{\rm rst}, \G_m) \ar[d] \\
H^2({\widehat M}, \G_m)\ar[r] & H^2(M^{\rm rst}, \G_m)}
\end{equation}
Since ${\widehat M}$ is also a desingularization of $M$, we conclude
by \cite[p. 309, Theorem 1]{Ni} (see also \cite[Theorem 1]{BHH}) that
$$
H^2({\widehat M}, \,\G_m)\,=\, 0\, .
$$
Hence from \eqref{dg} it follows that the image of $H^2({\widehat 
N}, \,\G_m)$ in $H^2(M^{\rm rst},\, \G_m)$ by the composition
in \eqref{c1} is zero. This completes the proof when $g\,\geq\, 3$.

Now assume that $g\,=\,2$ and rank $r=2$. So $M_{{\mathcal L}_X}$ is
already smooth. We take 
${\widetilde M}_{{\mathcal L}_X}\,=\, M_{{\mathcal L}_X}$. Let $M^{\rm free}
\,\subset\, M_{{\mathcal L}_X}$ be the largest Zariski
open subset where the action of $\Gamma$ is free. It follows from Remark
\ref{g2} that the complement of $M^{\rm free}$ is of codimension two. The
entire argument above works in this case after replacing $M^{\rm rst}$ by 
$M^{\rm free}$ and $N^{\rm rst}$ by $M^{\rm free}/\Gamma$.
\end{proof}

\end{document}